\documentclass[11pt]{amsart}
\usepackage[T1]{fontenc}
\usepackage{lmodern}
\usepackage{amsmath,amssymb,bm,microtype}
\usepackage{needspace}
\usepackage[hidelinks]{hyperref}
\newtheorem{theorem}{Theorem}[section]
\newtheorem{conjecture}[theorem]{Conjecture}
\newtheorem{lemma}[theorem]{Lemma}
\theoremstyle{remark}
\newtheorem{remark}[theorem]{Remark}
\numberwithin{equation}{section}
\newcommand{\Z}{\mathbb Z}
\newcommand{\Q}{\mathbb Q}
\newcommand{\R}{\mathbb R}
\newcommand{\Li}{\mathcal L_1}
\newcommand{\Lj}{\mathcal L_2}
\newcommand{\PP}{\bm P}
\newcommand{\QQ}{\bm Q}
\newcommand{\MM}{\mathcal M}
\title[Error terms for sums of two linear transformations]
{Error terms for sums of two linear transformations}
\author{Xiao-Hui Yan}
\address{School of Mathematics and Statistics, Anhui Normal University,
Wuhu 241002, P. R. China}
\email{yanxiaohui\_1992@163.com}
\date{}
\begin{document}
\begin{abstract}
Let $d$ be a positive integer, and let $\mathcal L_1,\mathcal L_2$ be
irreducible and coprime linear transformations from $\Z^d$ to
$\Z^d$. In this paper, we prove that, for every finite nonempty set
$A\subseteq\Z^d$,
\[
 |\mathcal L_1A+\mathcal L_2A|
 \ge \left(|\det\mathcal L_1|^{1/d}+|\det\mathcal L_2|^{1/d}\right)^d|A|-O\left(|A|^{1-1/d}\right).
\]
This confirms the conjecture of Conlon and Lim on the error term
for two linear transformations.
\end{abstract}
\maketitle

\section{Introduction}

For finite subsets $A,B$ of a vector space, let
\[
 A+B=\{\bm a+\bm b:\bm a\in A,\ \bm b\in B\}.
\]
For a linear transformation $\mathcal L$, let $\mathcal LA=\{\mathcal L\bm a:\bm a\in A\}$.
We write $|A|$ for the cardinality of $A$.
We identify a linear transformation with its matrix in the standard
basis, using the same matrix over $\Q$ or $\R$ as appropriate.
Invertibility of a transformation from $\Z^d$ to $\Z^d$ is understood
over $\Q$.

Bukh \cite[Theorem 2]{Bukh} proved that, for fixed integers
$\lambda_1,\ldots,\lambda_k$ with
$\gcd(\lambda_1,\ldots,\lambda_k)=1$,
\[
 |\lambda_1\cdot B+\cdots+\lambda_k\cdot B|
 \ge (|\lambda_1|+\cdots+|\lambda_k|)|B|-o(|B|)
\]
for every finite nonempty set $B\subseteq\Z$, where
$\lambda\cdot B=\{\lambda b:b\in B\}$. Later, Balog and Shakan
\cite[Theorem 1.1]{BS} improved the error term $o(|B|)$ to a constant depending only on $\lambda_1,\lambda_2$ when $k=2$
and $\lambda_1,\lambda_2$ are distinct positive coprime integers. Shakan
\cite[Theorem 1.1]{Shakan} extended this result to arbitrary
coprime integer coefficients $\lambda_1,\ldots,\lambda_k$.

In this paper, we consider sums of the form
$\mathcal L_1A+\cdots+\mathcal L_kA$, where $A\subseteq\Z^d$ and
$\mathcal L_1,\ldots,\mathcal L_k$ are linear transformations from
$\Z^d$ to $\Z^d$. Following the definitions of Conlon and Lim \cite[Definitions 1.2 and 1.3]{CL}, we call these linear transformations
$\mathcal L_1,\ldots,\mathcal L_k$
\emph{irreducible} if there are no subspaces $U,V\subseteq\Q^d$
with $0<\dim U=\dim V<d$ such that $\mathcal L_iU\subseteq V$ for
all $i$. They are \emph{coprime} if there are no invertible rational
matrices $\bm P,\bm Q$ with $0<|\det\bm P\det\bm Q|<1$ such that
$\bm P\mathcal L_i\bm Q$ has integer entries
for every $i$. Conlon and Lim {\cite[Conjecture 1.4]{CL}} posed the following conjecture.

\begin{conjecture}\cite[Conjecture 1.4]{CL}
\label{conj:CL-asymptotic}
Let $\mathcal L_1,\ldots,\mathcal L_k$ be irreducible and coprime
linear transformations from $\Z^d$ to $\Z^d$.
Then, for every finite nonempty set $A\subseteq\Z^d$,
\[
 |\mathcal L_1A+\cdots+\mathcal L_kA|
 \ge\left(\sum_{i=1}^k|\det\mathcal L_i|^{1/d}\right)^d|A|-o(|A|).
\]
\end{conjecture}

This is a modified version of a conjecture first posed on Bukh's
webpage and later reiterated by several other authors
\cite{KP2020,Mudgal,Shakan}. The coefficient is suggested by the Brunn--Minkowski inequality;
see \cite{GardnerBM}. Conlon and Lim \cite[Section 1]{CL} showed the conditions are both necessary.

Conlon and Lim \cite[Theorem 1.5]{CL} proved
Conjecture~\ref{conj:CL-asymptotic} for two transformations.

\begin{theorem}\cite[Theorem 1.5]{CL}
\label{thm:Conlon-Lim}
Let $\mathcal L_1,\mathcal L_2$ be irreducible and coprime
linear transformations from $\Z^d$ to $\Z^d$.
Then there are constants $D,\sigma>0$ such that, for any finite subset $A\subseteq\Z^d$,
\[
 |\mathcal L_1A+\mathcal L_2A|
 \ge\left(|\det\mathcal L_1|^{1/d}+|\det\mathcal L_2|^{1/d}\right)^d|A|
      -D|A|^{1-\sigma}.
\]
\end{theorem}
Lopez Bruch, Jing and Mudgal \cite{LBJM} obtained related bounds
for sums of several linear transformations. More recently, Conlon and Lim \cite[Corollary 1.8]{Algebraic}
determined the optimal leading coefficient for sums of two
irreducible, coprime transformations, with an $o(|A|)$ error term.

Conlon and Lim \cite[Conjecture~6.2]{CL} and Shakan \cite[Question~1.2]{Shakan} all conjectured that the error term may be $O\left(|A|^{1-1/d}\right)$. In this paper, we obtain this error term for two transformations.

\begin{theorem}\label{thm:main}
Let $\mathcal L_1,\mathcal L_2$ be irreducible and coprime
linear transformations from $\Z^d$ to $\Z^d$.
Then
\[
 |\Li A+\Lj A|\ge\left(|\det\mathcal L_1|^{1/d}+|\det\mathcal L_2|^{1/d}\right)^d|A|-O\left(|A|^{1-1/d}\right)
\]
for every finite nonempty set $A\subseteq\Z^d$.
\end{theorem}

\section{Geometric and combinatorial estimates}
We first improve the exponent
in the bound of Conlon and Lim \cite[Lemma 2.4]{CL}.

\begin{lemma}\label{lem:hyperplane}
Let $K\ge1$ and let $\mathcal T$ be an invertible linear transformation
of $\Q^d$ with no nonzero proper invariant subspace over $\Q$.
Let $A\subseteq\Q^d$ be finite and nonempty. If
$|A+\mathcal TA|\le K|A|$, then
\[
 |A\cap W|\le K^{(d+1)/2}|A|^{1-1/d}
\]
for every affine hyperplane $W$ in $\mathbb{R}^d$.
\end{lemma}

\begin{proof}
Let $B=A\cap W$. If $B=\varnothing$, the result is true.
Otherwise, fix $\bm b_0\in B$ and let
\[
 U=\operatorname{span}_{\Q}
 \{\bm b-\bm b_0:\bm b\in B\}\subseteq W-\bm b_0.
\]
Then $\dim_{\Q}U\le d-1$.
Hence there is a nonzero linear functional $\ell:\Q^d\to\Q$
such that, for every $\bm b\in B$,
\[
0=\ell(\bm b-\bm b_0)
 =\ell(\bm b)-\ell(\bm b_0),
\]
that is,
\begin{equation}\label{eq:lB}
\ell(\bm b)=\ell(\bm b_0),\quad \bm b\in B.
\end{equation}

For each integer $i\ge0$, let
\[
 R_i=\{\ell(\mathcal T^i\bm a):\bm a\in A\},
 \qquad r_i=|R_i|,
\]
and choose $c_i\in R_i$ such that the set
\[
 F_i=\{\bm a\in A:\ell(\mathcal T^i\bm a)=c_i\}
\]
has maximum cardinality.

Fix an integer $i\ge0$. For each $t\in R_{i+1}$,
choose $\bm a_t\in A$ such that
$\ell(\mathcal T^{i+1}\bm a_t)=t$.
For distinct $s,t\in R_{i+1}$ and any $\bm f,\bm g\in F_i$,
we have
\[
 \ell\bigl(\mathcal T^i(\bm f+\mathcal T\bm a_t)\bigr)
 =c_i+t\ne c_i+s
 =\ell\bigl(\mathcal T^i(\bm g+\mathcal T\bm a_s)\bigr).
\]
Thus the sets $F_i+\mathcal T\bm a_t$ with $t\in R_{i+1}$
are pairwise disjoint subsets of $A+\mathcal TA$.
Since $r_{i+1}=|R_{i+1}|$, it follows that
\begin{equation}\label{eq:FiKA}
 |F_i|r_{i+1}\le |A+\mathcal TA|\le K|A|.
\end{equation}

By \eqref{eq:lB} and the definition of $F_i$, we have
\[
|B|
\le\bigl|\{\bm a\in A:\ell(\bm a)=\ell(\bm b_0)\}\bigr|
\le |F_0|
\]
and
\[
|A|
=\sum_{c\in R_i}
\bigl|\{\bm a\in A:\ell(\mathcal T^i\bm a)=c\}\bigr|
\le r_i|F_i|.
\]
It follows from \eqref{eq:FiKA} that
\[
 r_1\le\frac{K|A|}{|F_0|}\le \frac{K|A|}{|B|},\qquad r_{i+1}\le \frac{K|A|}{|F_i|}\le Kr_i.
\]
Hence,
\[
 r_i\le\frac{K^i|A|}{|B|}\qquad(i\ge 1).
\]

Let
\[
 V=\bigcap_{j=0}^{d-1}\ker(\ell\mathcal T^j).
\]
We first show that $V$ is $\mathcal T$-invariant.
By the Cayley--Hamilton theorem, there exist
$a_0,\ldots,a_{d-1}\in\Q$ such that
\[
 \mathcal T^d=-\sum_{j=0}^{d-1}a_j\mathcal T^j.
\]
For any $\bm v\in V$, we have
$\ell(\mathcal T^j\bm v)=0$ for $0\le j\le d-1$, and hence
\[
 \ell(\mathcal T^d\bm v)
 =-\sum_{j=0}^{d-1}a_j\ell(\mathcal T^j\bm v)=0.
\]
It follows that
\[
 \ell\bigl(\mathcal T^j(\mathcal T\bm v)\bigr)
 =\ell(\mathcal T^{j+1}\bm v)=0
 \qquad (0\le j\le d-1),
\]
so $\mathcal T\bm v\in V$.
Thus $V$ is a rational $\mathcal T$-invariant subspace.
Since $\ell\ne0$, we have $V\subseteq\ker\ell\ne\Q^d$.
Since $\mathcal T$ has no nonzero proper invariant subspace, it follows that $V=\{\bm 0\}$.

Now consider the linear map
\[
 \Phi:\Q^d\longrightarrow\Q^d,\qquad
 \Phi(\bm x)=
 \bigl(\ell(\mathcal T\bm x),\ldots,
       \ell(\mathcal T^d\bm x)\bigr).
\]
If $\Phi(\bm x)=\bm 0$, then $\mathcal T\bm x\in V=\{\bm 0\}$.
Since $\mathcal T$ is invertible, this implies $\bm x=\bm 0$.
Hence $\Phi$ is injective.
Moreover, $\Phi(A)\subseteq R_1\times\cdots\times R_d$, so
\[
 |A|=|\Phi(A)|
 \le \prod_{i=1}^d r_i
 \le \prod_{i=1}^d\frac{K^i|A|}{|B|}
 =\frac{K^{d(d+1)/2}|A|^d}{|B|^d},
\]
that is,
\[
 |A\cap W|=|B|\le K^{(d+1)/2}|A|^{1-1/d}.
\]
This completes the proof of Lemma \ref{lem:hyperplane}.
\end{proof}

For earlier discrete versions of the Brunn--Minkowski inequality,
see Gardner and Gronchi \cite{GG} and Green and Tao \cite{GT}. Now we apply Lemma~\ref{lem:hyperplane} and the discrete
Brunn--Minkowski inequality of
van Hintum, Keevash and Tiba \cite[Theorem 1.2]{VKT} to improve the error exponent in the bound of Conlon and Lim \cite[Lemma 2.5]{CL}.

\begin{lemma}\label{lem:mixed}
Let $\Li,\Lj$ be irreducible and coprime linear transformations from
$\Z^d$ to $\Z^d$, and let
\[
 p=|\det\Li|,\qquad q=|\det\Lj|,\qquad
 H=(p^{1/d}+q^{1/d})^d.
\]
Let $A\subseteq\Z^d$ be finite and nonempty.
Suppose that $|\Li A+\Lj A|\le H|A|$.
For each $\rho>0$, there is a constant $K=K(d,p,q,\rho)$ such that
\begin{equation}\label{eq:mixed}
 |X+Y|\ge\bigl(|X|^{1/d}+|Y|^{1/d}\bigr)^d-K|A|^{1-1/d}
\end{equation}
whenever $X,Y\subseteq\Z^d$ are invertible affine images of subsets
of $A$ and $|X|,|Y|\ge\rho |A|$.
\end{lemma}

\begin{proof}
For $d=1$, the conclusion follows from
$|X+Y|\ge |X|+|Y|-1$. Next, we assume that $d\ge2$.
Since $\Li,\Lj$ are irreducible, $\Li^{-1}\Lj$ has no nonzero
proper invariant subspace over $\Q$.
Since $\mathcal L_1$ is invertible, we have
\[
 |A+\mathcal L_1^{-1}\mathcal L_2A|
 =|\mathcal L_1A+\mathcal L_2A|
 \le H|A|.
\]
Applying Lemma~\ref{lem:hyperplane} with
$\mathcal T=\mathcal L_1^{-1}\mathcal L_2$ and $K=H$, we obtain
\[
 |A\cap W|\le H^{(d+1)/2}|A|^{1-1/d}
\]
for every real affine hyperplane $W$.
Moreover, if $A'\subseteq A$ and $\phi$ is an invertible
affine transformation of $\mathbb R^d$, then
$\phi^{-1}(W)$ is also a real affine hyperplane, so
\[
 |\phi(A')\cap W|
 =|A'\cap\phi^{-1}(W)|
 \le H^{(d+1)/2}|A|^{1-1/d}.
\]

We may assume that $|X|\ge|Y|$. For sufficiently large $|A|$, let
\[
 m=\left\lfloor\frac{|Y|}{2H^{(d+1)/2}|A|^{1-1/d}}\right\rfloor.
\]
Then
\[
 m\ge\frac{\rho}{4H^{(d+1)/2}}|A|^{1/d},\qquad
 mH^{(d+1)/2}|A|^{1-1/d}\le\frac{|Y|}{2}.
\]
If $Y$ can be covered by $m$ parallel hyperplanes $W_1,\cdots,W_m$, then
\[
|Y|
\le\sum_{j=1}^{m}|Y\cap W_j|
\le mH^{(d+1)/2}|A|^{1-1/d}
\le\frac{|Y|}{2},
\]
a contradiction.
Hence $Y$ cannot be covered by $m$ parallel hyperplanes.
By \cite[Theorem 1.2]{VKT}, with dimension $d$ and $t=1$, we have
\begin{align}\label{eq:X+Y}
 |X+Y|^{1/d}
 &\ge |X|^{1/d}+|Y|^{1/d}-O_d\bigl(|Y|^{1/d}/m\bigr)\notag\\
 &\ge |X|^{1/d}+|Y|^{1/d}-O_{d,p,q,\rho}(1),
\end{align}
where the last inequality follows from
\[
\frac{|Y|^{1/d}}m
\le
\frac{|A|^{1/d}}
{\dfrac{\rho}{4H^{(d+1)/2}}|A|^{1/d}}
=
\frac{4H^{(d+1)/2}}{\rho}.
\]
Thus, there is a constant $C=C(d,p,q,\rho)>0$ such that
\[
 |X+Y|^{1/d}\ge |X|^{1/d}+|Y|^{1/d}-C
\]
for all sufficiently large $|A|$. Choose $N_0=N_0(d,p,q,\rho)\ge1$ sufficiently large that
this estimate holds whenever $|A|\ge N_0$ and
\[
 2\rho^{1/d}N_0^{1/d}>C.
\]
For $|A|\ge N_0$, we have
\[
 |X|^{1/d}+|Y|^{1/d}-C
 \ge 2\rho^{1/d}|A|^{1/d}-C>0.
\]
By the mean value theorem and the bounds $|X|,|Y|\le |A|$,
\begin{align*}
 |X+Y|
 &\ge \bigl(|X|^{1/d}+|Y|^{1/d}-C\bigr)^d\\
 &\ge \bigl(|X|^{1/d}+|Y|^{1/d}\bigr)^d
      -dC\bigl(|X|^{1/d}+|Y|^{1/d}\bigr)^{d-1}\\
 &\ge \bigl(|X|^{1/d}+|Y|^{1/d}\bigr)^d
      -d\,2^{d-1}C|A|^{1-1/d}.
\end{align*}
It remains to consider $1\le |A|<N_0$.
In this case,
\[
 \bigl(|X|^{1/d}+|Y|^{1/d}\bigr)^d
 \le 2^d|A|
 \le 2^dN_0^{1/d}|A|^{1-1/d}.
\]
Taking
\[
 K=\max\bigl\{d\,2^{d-1}C,\;2^dN_0^{1/d}\bigr\},
\]
we see that the right-hand side of \eqref{eq:mixed}
is nonpositive when $|A|<N_0$, so \eqref{eq:mixed}
holds in this case as well. This completes the proof of Lemma \ref{lem:mixed}.
\end{proof}

For the first inequality in the next lemma, we use the argument of
Conlon and Lim \cite[proof of Lemma 4.2]{CL}.

\begin{lemma}\label{lem:array}
Let $d\ge2$ and $p,q\ge1$ be integers, and let
\[
 H=(p^{1/d}+q^{1/d})^d.
\]
Let $(c_{ij})$ be a $p\times q$ array of nonnegative real numbers with
\[
 n=\sum_{i,j}c_{ij}>0.
\]
Let
\[
 a_i=\sum_jc_{ij},\qquad b_j=\sum_ic_{ij},\qquad G=\max_{i,j}\bigl\{(a_i^{1/d}+b_j^{1/d})^d-Hc_{ij}\bigr\}.
\]
Then $G\ge0$ and
\begin{equation}\label{eq:array}
 \begin{split}
\max_{i,j}c_{ij}-\min_{i,j}c_{ij}&\le(H+1)G,\\
 \sum_{i,j}(a_i^{1/d}+b_j^{1/d})^d&\ge Hn-MG^2/n,
 \end{split}
\end{equation}
where $M>0$ depends only on $d,p,q$.
Zero rows and columns are allowed.
\end{lemma}

\begin{proof}
By permuting rows and columns, we may assume that
$m=\min_{i,j}c_{ij}=c_{11}$. Since $a_i\ge qm$ and $b_j\ge pm$,
the identity
\[
 (a_i^{1/d}+b_j^{1/d})^d
 =a_i+b_j+\sum_{r=1}^{d-1}\binom dr a_i^{r/d}b_j^{1-r/d}
\]
gives
\begin{align*}
(a_i^{1/d}+b_j^{1/d})^d
&\ge a_i+b_j+
  m\sum_{r=1}^{d-1}\binom dr q^{r/d}p^{1-r/d}\\
&=Hm+(a_i-qm)+(b_j-pm).
\end{align*}
For $(i,j)=(1,1)$, this gives $G\ge0$ and $a_1-qm\le G$.
Hence $c_{1j}\le m+G$ for every $j$. With $i=1$, we then have
\[
 b_j-pm\le H(c_{1j}-m)+G\le(H+1)G.
\]
Consequently,
\[
 m\le c_{ij}\le m+(H+1)G
\]
for all $i,j$. This proves the first inequality in \eqref{eq:array}.

Let $0<s<1$. By the convexity of $t^{s-1}$ for $t>0$, we have
\[
 t^s\ge(2-s)t-(1-s)t^2
      =1+s(t-1)-(1-s)(t-1)^2.
\]
By continuity, this also holds at $t=0$.
Apply the inequality with $t=pa_i/n$ and sum over $i$.
Together with the concavity of $t^s$, this gives
\[
 0\le p^{1-s}n^s-\sum_i a_i^s
 \le(1-s)(n/p)^{s-2}\sum_i(a_i-n/p)^2.
\]
Similarly,
\[
 0\le q^sn^{1-s}-\sum_j b_j^{1-s}
 \le s(n/q)^{-1-s}\sum_j(b_j-n/q)^2.
\]
It follows that
\begin{align*}
 p^{1-s}q^sn-
 \Bigl(\sum_i a_i^s\Bigr)\Bigl(\sum_jb_j^{1-s}\Bigr)
 &\le \frac{(1-s)p^{2-s}q^s}{n}\sum_i(a_i-n/p)^2\\
 &\quad+\frac{sp^{1-s}q^{1+s}}{n}\sum_j(b_j-n/q)^2.
\end{align*}
By the first inequality in \eqref{eq:array}, we have
\begin{align*}
 \sum_i(a_i-n/p)^2&\le pq^2(H+1)^2G^2,\\
 \sum_j(b_j-n/q)^2&\le qp^2(H+1)^2G^2.
\end{align*}
Taking $s=r/d$, multiplying by $\binom dr$ and summing over
$1\le r\le d-1$, we obtain the second inequality in
\eqref{eq:array}, since
\[
 \sum_{i,j}(a_i+b_j)=(p+q)n
\]
and
\[
 p+q+\sum_{r=1}^{d-1}\binom dr p^{1-r/d}q^{r/d}=H.
\]
This completes the proof.
\end{proof}

\section{Proof of Theorem~\ref{thm:main}}\label{sec:pair}
When $d=1$, write $\mathcal L_i x=\lambda_i x$ with
$\lambda_i\in\Z$. Coprimeness gives
$\gcd(\lambda_1,\lambda_2)=1$.
If one coefficient is zero, the other is $\pm1$, and the
conclusion is immediate. Otherwise, the conclusion follows
from \cite[Theorem 1.1]{Shakan}.
We therefore assume throughout the rest of this section
that $d\ge2$. Irreducibility then implies that both
$\mathcal L_1$ and $\mathcal L_2$ are invertible over $\Q$.
Let $\Li,\Lj$ be irreducible and coprime linear transformations from
$\Z^d$ to $\Z^d$.
Throughout this section, let
\[
 p=|\det\Li|,\qquad q=|\det\Lj|,\qquad
 H=(p^{1/d}+q^{1/d})^d.
\]

For a full-rank sublattice $\Lambda$ of $\Z^d$, write $\Z^d/\Lambda$
for its quotient group and $[\Z^d:\Lambda]$ for its index. Let
\begin{equation}\label{eq:lattices}
 \begin{split}
 K_0&=\Li\Z^d\cap\Lj\Z^d,\\
 P_1&=\Z^d\cap\Lj^{-1}\Li\Z^d,\qquad
 P_2=\Z^d\cap\Li^{-1}\Lj\Z^d,\qquad P_0=P_1\cap P_2.
 \end{split}
\end{equation}
By \cite[Section 4]{CL},
\begin{equation}\label{eq:indices}
 \begin{gathered}
 \relax[\Z^d:P_1]=p,\qquad [\Z^d:P_2]=q,\\
 [\Z^d:P_0]=[\Z^d:K_0]=pq,\qquad P_1+P_2=\Z^d.
 \end{gathered}
\end{equation}
Let $\PP,\QQ$ be $d\times d$ integer matrices whose columns form $\Z$-bases
of $P_0,K_0$, respectively. The linear transformations
\begin{equation}\label{eq:M}
 \MM_i=\QQ^{-1}\mathcal L_i\PP\qquad(i=1,2)
\end{equation}
have integer entries and form an irreducible, coprime pair.
Moreover, $|\det\MM_1|=p$ and $|\det\MM_2|=q$; see
\cite[proof of Lemma 4.5, Case 2]{CL}.

For $A\subseteq\Z^d$, let $A_1,\ldots,A_p$ and $A^1,\ldots,A^q$
be its intersections with the cosets of $P_1$ and $P_2$, respectively,
including empty intersections. Let
\[
 A_{ij}=A_i\cap A^j.
\]
Call $A_i$ and $A^j$ the row and column parts, respectively.
By \eqref{eq:indices}, these $pq$ intersections are the parts of $A$
modulo $P_0$. The nonempty sets $\Li A^j+\Lj A_i$ lie in distinct
cosets of $K_0$; see \cite[proof of Lemma 4.2]{CL}. Hence
\begin{equation}\label{eq:partition}
 |\Li A+\Lj A|=\sum_{i,j}|\Li A^j+\Lj A_i|,
\end{equation}
where a summand is zero if either set is empty. Since
$A_{kl}= A_k\cap A^l$, we also have
\begin{equation}\label{eq:replace}
 |\Li A+\Lj A|\ge |\Li A^j+\Lj A_i|
 +\sum_{(k,l)\ne(i,j)}|\Li A_{kl}+\Lj A_{kl}|
\end{equation}
whenever $A_i,A^j\ne\varnothing$.

\Needspace{16\baselineskip}
The argument in \cite[proof of Lemma 4.2]{CL} gives the following
lemma. We include the proof to specify the constants.

\begin{lemma}\label{lem:small}
Let $1-1/d\le\beta<1$ and $D>0$. Suppose that
\begin{equation}\label{eq:old}
 |\Li B+\Lj B|\ge H|B|-D|B|^\beta
\end{equation}
holds uniformly for all finite $B\subseteq\Z^d$. Let
\[
 \delta=\frac1{8\max(p,q)H^2}.
\]
If $A_i,A^j\ne\varnothing$ and $|A_{ij}|\le\delta|A|$ for some
$i,j$, then
\begin{equation}\label{eq:small}
 |\Li A+\Lj A|\ge H|A|+H\delta|A|-Dpq|A|^\beta.
\end{equation}
Here $A_{ij}$ may be empty.
\end{lemma}

\begin{proof}
If $|A^j|\ge2H\delta|A|$, then
\[
 |\Li A^j+\Lj A_i|-H|A_{ij}|
 \ge |A^j|-H|A_{ij}|\ge H\delta|A|.
\]
If $|A^j|<2H\delta|A|$, choose $i'$ with $|A_{i'}|\ge|A|/p$.
Since $|A_{i'j}|\le|A^j|<2H\delta|A|$, we have
\begin{align*}
 |\Li A^j+\Lj A_{i'}|-H|A_{i'j}|
 &\ge |A_{i'}|-H|A_{i'j}|\\
 &>|A|/p-2H^2\delta|A|
 \ge 3|A|/(4p)\ge H\delta|A|.
\end{align*}
In either case, use the corresponding mixed sum in \eqref{eq:replace} and apply
\eqref{eq:old} to the remaining intersections. There are at most
$pq$ such intersections, each of size at most $|A|$, so
\[
 |\Li A+\Lj A|\ge H|A|+H\delta|A|-Dpq|A|^\beta.
\]
\end{proof}

For the next lemma, we use the lattice reduction in
\cite[proof of Lemma 4.5]{CL}, together with
Lemmas~\ref{lem:mixed} and~\ref{lem:array}.

\begin{lemma}\label{lem:bootstrap}
Suppose that \eqref{eq:old} holds uniformly for some $D>0$ and
$1-1/d\le\beta<1$. Then there is a constant
$D'=D'(d,p,q,D,\beta)$ such that
\begin{equation}\label{eq:new}
 |\Li A+\Lj A|\ge H|A|-D'|A|^{\max\{2\beta-1,1-1/d\}},
\end{equation}
for every such pair and every finite nonempty
$A\subseteq\Z^d$.
\end{lemma}

\begin{proof}
Suppose that, for some $D>0$ and $1-1/d\le\beta<1$,
\eqref{eq:old} holds with the same constant $D$ for every
finite $B\subseteq\Z^d$ and every irreducible, coprime pair
$\Li,\Lj$ satisfying
$|\det\Li|=p$ and $|\det\Lj|=q$.

Let $\delta$ be as in Lemma~\ref{lem:small} and let $N=|A|$.
We may assume that
\[
 |\Li A+\Lj A|<HN,\qquad N\ge N_0,
\]
where $N_0=N_0(d,p,q,D,\beta)$ is sufficiently large.
Take $N_0>H^{d(d+1)/2}$. By Lemma~\ref{lem:hyperplane},
applied with $\mathcal T=\Li^{-1}\Lj$ and $K=H$, the set $A$ is not
contained in an affine hyperplane.

By translation, we may assume that $\bm0\in A$. Consider all triples
obtained by changing coordinates separately in the domain and range:
\[
 (\bm S^{-1}\Li\bm R,\ \bm S^{-1}\Lj\bm R,\ \bm R^{-1}A),
\]
where $\bm R,\bm S$ are invertible $d\times d$ rational matrices,
$|\det\bm R|=|\det\bm S|$, both $\bm S^{-1}\Li\bm R$ and $\bm S^{-1}\Lj\bm R$
have integer entries,
and $\bm R^{-1}A\subseteq\Z^d$.
The transformed pair is irreducible and coprime, and $p,q,N$ and
$|\Li A+\Lj A|$ are unchanged. Since $A$ spans $\Q^d$, the subgroup
$\langle A\rangle_{\Z}$ generated by $A$ has finite index in $\Z^d$.
Choose a triple for which this index is least, and retain the same
notation. Define the corresponding lattices and basis matrices by
\eqref{eq:lattices}--\eqref{eq:M}.

Denote the nonempty sets $A_{ij}$ by $A_t$. If
$0<|A_t|\le\delta N$ for some $t$, then Lemma~\ref{lem:small} gives
\[
 |\Li A+\Lj A|\ge HN+H\delta N-DpqN^\beta\ge HN
\]
for sufficiently large $N$, a contradiction. Thus
\begin{equation}\label{eq:heavy-root}
 |A_t|>\delta N\qquad\hbox{for every }t.
\end{equation}

Let
\begin{equation}\label{eq:fine-lattices}
 \Lambda_1=P_0\cap\Lj^{-1}\Li P_0,\qquad
 \Lambda_2=P_0\cap\Li^{-1}\Lj P_0.
\end{equation}
By \cite[Lemma 4.3]{CL}, applied also with the transformations exchanged,
\[
 [\Z^d:\Lambda_1]=p^2q,\qquad [\Z^d:\Lambda_2]=pq^2.
\]
For $h,i\in\{1,2\}$, define
\[
 \phi_{h,i}:\Z^d/\Lambda_h\longrightarrow\Z^d/(\mathcal L_hP_0),
 \qquad
 \phi_{h,i}(\bm u+\Lambda_h)=\mathcal L_i\bm u+\mathcal L_hP_0.
\]
By \cite[Lemma 4.3]{CL}, with the roles of $\mathcal L_1$
and $\mathcal L_2$ interchanged when $h=2$, these
homomorphisms are well defined, and
$\phi_h=\phi_{h,1}+\phi_{h,2}$ is an isomorphism.

For each $t$, choose $\bm v_t\in A_t$ and let
\begin{equation}\label{eq:B}
 B_t=\PP^{-1}(A_t-\bm v_t).
\end{equation}
Then $B_t\subseteq\Z^d$, $|B_t|=|A_t|$, and
\[
 |\Li A_t+\Lj A_t|=|\MM_1B_t+\MM_2B_t|.
\]
The row and column lattices for $\MM_1,\MM_2$ satisfy
\begin{equation}\label{eq:local-marginals}
 \Z^d\cap\MM_2^{-1}\MM_1\Z^d=\PP^{-1}\Lambda_1,\qquad
 \Z^d\cap\MM_1^{-1}\MM_2\Z^d=\PP^{-1}\Lambda_2.
\end{equation}

Every nonempty part of $A$ modulo either $\Lambda_h$ has size greater
than $\delta^2N$. For otherwise, a nonempty part contained in some
$A_t$ has size at most $\delta^2N$. By
\eqref{eq:local-marginals}, it corresponds to a nonempty row or column
part of $B_t$. It contains a nonempty row--column intersection of size at most
$\delta^2N<\delta|A_t|$. Lemma~\ref{lem:small} gives
\[
 |\Li A_t+\Lj A_t|
 \ge H|A_t|+H\delta|A_t|-Dpq|A_t|^\beta.
\]
The sets $\Li A_s+\Lj A_s$ lie in distinct cosets of $K_0$.
Applying \eqref{eq:old} to the other parts and using
\eqref{eq:heavy-root}, we obtain
\[
 |\Li A+\Lj A|\ge HN+H\delta^2N-2DpqN^\beta\ge HN
\]
for sufficiently large $N$, a contradiction. Hence
\begin{equation}\label{eq:heavy-fine}
 |A\cap(\bm u+\Lambda_h)|>\delta^2N
 \quad\hbox{if }A\cap(\bm u+\Lambda_h)\ne\varnothing,
 \qquad h=1,2.
\end{equation}

Let $X_h$ be the image of $A$ in $\Z^d/\Lambda_h$. We have
\[
 \phi_h(X_h)\subseteq\phi_{h,1}(X_h)+\phi_{h,2}(X_h).
\]
Suppose that the set on the right has size greater than $|X_h|$.
Since $\phi_h$ is injective, there are two nonempty parts $D_1,D_2$
of $A$ modulo $\Lambda_h$ such that $\Li D_1+\Lj D_2$ lies in a
coset of $\mathcal L_hP_0$ different from all those containing
$\Li D+\Lj D$, where $D$ runs through these parts. The latter cosets
are distinct. By \eqref{eq:heavy-fine},
\[
 |\Li D_1+\Lj D_2|>\delta^2N.
\]
There are at most $pq\max(p,q)$ parts modulo $\Lambda_h$, so
\eqref{eq:old} gives
\[
 |\Li A+\Lj A|\ge HN+\delta^2N-Dpq\max(p,q)N^\beta\ge HN
\]
for sufficiently large $N$, a contradiction. Therefore
\begin{equation}\label{eq:no-growth}
 \phi_{h,1}(X_h)+\phi_{h,2}(X_h)=\phi_h(X_h)
 \qquad(h=1,2).
\end{equation}
Choose $N_0$ large enough for the preceding three estimates to give
contradictions. Since $\beta<1$, this choice depends only on $d,p,q,D,\beta$.

We now show that $X_h$ generates $\Z^d/\Lambda_h$, using the
normalization argument in \cite[proof of Lemma 4.5, Case 1]{CL}.
Suppose otherwise, and let
\[
 E=\langle A\rangle_{\Z}+\Lambda_h,\qquad
 E'=\langle\Li A+\Lj A\rangle_{\Z}+\mathcal L_hP_0.
\]
By \eqref{eq:no-growth},
\[
 E/\Lambda_h=\langle X_h\rangle,\qquad
 E'/(\mathcal L_hP_0)=\phi_h(\langle X_h\rangle).
\]
Since $\phi_h$ is an isomorphism, we have
\[
 [\Z^d:E]=[\Z^d:E']>1.
\]
Moreover, $\Li E,\Lj E\subseteq E'$, since $\bm0\in A$ gives
$\mathcal L_iA\subseteq\Li A+\Lj A$ and, by definition,
$\mathcal L_i\Lambda_h\subseteq\mathcal L_hP_0$.
Let $\bm R,\bm S$ have columns forming $\Z$-bases of $E,E'$, respectively.
Then
\[
 (\bm S^{-1}\Li\bm R,\ \bm S^{-1}\Lj\bm R,\ \bm R^{-1}A)
\]
is one of the triples considered above, but
\[
 [\Z^d:\langle\bm R^{-1}A\rangle_{\Z}]
 =[E:\langle A\rangle_{\Z}]
 =\frac{[\Z^d:\langle A\rangle_{\Z}]}{[\Z^d:E]}
 <[\Z^d:\langle A\rangle_{\Z}].
\]
This contradicts the choice of the triple.

Let $x\in X_h$. Since $0\in X_h$, the set $X_h-x$ contains both
$0$ and $-x$. The subgroup it generates contains $x$ and hence every element of
$X_h$, so it is $\Z^d/\Lambda_h$. By
\eqref{eq:no-growth},
\[
 |\phi_{h,1}(X_h-x)+\phi_{h,2}(X_h-x)|=|X_h-x|.
\]
Apply \cite[Lemma 4.4]{CL} to $X_h-x$, exchanging the transformations when
$h=2$. The first two alternatives are excluded, and hence
\begin{equation}\label{eq:all-fibres}
 x+(P_0/\Lambda_h)\subseteq X_h
 \qquad(x\in X_h,\ h=1,2).
\end{equation}
By \eqref{eq:local-marginals} and \eqref{eq:all-fibres}, each
$B_t$ meets all $p$ row cosets and all $q$ column cosets for
$\MM_1,\MM_2$. By \eqref{eq:heavy-fine}, each of these
row and column parts has size greater than $\delta^2N$.

Fix $t$, and let $B_{t,i}$ and $B_t^j$ be the row and column parts
of $B_t$. Put
\[
 c_{ij}=|B_{t,i}\cap B_t^j|,\qquad
 a_i=|B_{t,i}|,\qquad b_j=|B_t^j|,
\]
\[
 G=\max_{i,j}\{(a_i^{1/d}+b_j^{1/d})^d-Hc_{ij}\}.
\]
The sets
$\MM_1B_t^j$ and $\MM_2B_{t,i}$ lie in $\Z^d$ and are invertible
affine images of subsets of $A$. By Lemma~\ref{lem:mixed},
with $\rho=\delta^2$, there is a constant $K=K(d,p,q)$
such that
\[
 |\MM_1B_t^j+\MM_2B_{t,i}|\ge (a_i^{1/d}+b_j^{1/d})^d-KN^{1-1/d}
 \qquad(1\le i\le p,\ 1\le j\le q).
\]
Applying \eqref{eq:partition} to $\MM_1,\MM_2,B_t$ and using
Lemma~\ref{lem:array}, we obtain
\begin{equation}\label{eq:quadratic}
 |\MM_1B_t+\MM_2B_t|\ge H|B_t|-MG^2/|B_t|-pqKN^{1-1/d}.
\end{equation}
Choose $(i,j)$ such that $G=(a_i^{1/d}+b_j^{1/d})^d-Hc_{ij}$. By
\eqref{eq:replace}, using this mixed sum and applying \eqref{eq:old}
to the remaining intersections, we also have
\begin{equation}\label{eq:linear}
 |\MM_1B_t+\MM_2B_t|\ge H|B_t|+G-DpqN^\beta-KN^{1-1/d}.
\end{equation}
This holds even if $c_{ij}=0$, since the corresponding row and
column parts are nonempty.

If $G\ge DpqN^\beta+KN^{1-1/d}$, then \eqref{eq:linear} gives
\[
 |\MM_1B_t+\MM_2B_t|\ge H|B_t|.
\]
If $G<DpqN^\beta+KN^{1-1/d}$, then $G<(Dpq+K)N^\beta$, since
$\beta\ge1-1/d$. By \eqref{eq:quadratic} and $|B_t|>\delta N$, we have
\[
 |\MM_1B_t+\MM_2B_t|\ge H|B_t|
 -\frac{M(Dpq+K)^2}{\delta}N^{2\beta-1}-pqKN^{1-1/d}.
\]
Thus the latter bound holds in either case. Summing over $t$ and
using the disjointness of $\Li A_t+\Lj A_t$, we obtain
\[
 |\Li A+\Lj A|\ge HN
 -\frac{pqM(Dpq+K)^2}{\delta}N^{2\beta-1}-(pq)^2KN^{1-1/d}.
\]
Thus \eqref{eq:new} holds for $N\ge N_0$ if
\[
 D'\ge\frac{pqM(Dpq+K)^2}{\delta}+(pq)^2K.
\]
Taking also $D'\ge HN_0^{1/d}$, we obtain the same bound for
$1\le N<N_0$. The constant $D'$ depends only on $d,p,q,D,\beta$.
\end{proof}

\begin{proof}[Proof of Theorem~\ref{thm:main}]
By \cite[Theorem 4.6]{CL}, \eqref{eq:old} holds uniformly with
$\beta=1-\sigma$, where $D>0$ and $\sigma>0$ depend only on $d,p,q$.
By decreasing $\sigma$ if necessary, we may assume that $\sigma\le1/d$.
After $r$ applications of Lemma~\ref{lem:bootstrap}, the error exponent is
\[
 \max\{1-2^r\sigma,1-1/d\}.
\]
Choose $r$ such that $2^r\sigma\ge1/d$. We obtain
\[
 |\Li A+\Lj A|\ge H|A|-C_{d,p,q}|A|^{1-1/d}.
\]
At each application, the new constant depends only on $d,p,q$ and
the preceding constant and exponent. The number of applications also
depends only on $d,p,q$. Thus $C_{d,p,q}$ depends only on $d,p,q$.
\end{proof}

\begin{remark}
Theorem~\ref{thm:main} holds with the same constant $C_{d,p,q}$ for
finite nonempty $A\subseteq\R^d$. By the coordinatewise Freiman
embedding in \cite[proof of Theorem 6.3]{CL}, with order sufficiently
large in terms of the coefficients of $\Li,\Lj$, there is a set $A'\subseteq\Z^d$
such that
\[
 |A'|=|A|,\qquad |\Li A'+\Lj A'|=|\Li A+\Lj A|.
\]
Apply Theorem~\ref{thm:main} to $A'$. The Freiman embedding used here
is given in \cite[Lemma 5.25]{TV}.
\end{remark}

\end{document}